\documentclass{amsart}
\makeatletter
\@namedef{subjclassname@2020}{%
  \textup{2020} Mathematics Subject Classification}
\makeatother
\usepackage{amssymb}
\usepackage[utf8]{inputenc} 
\usepackage[T1]{fontenc}    
\usepackage{url}            
\usepackage{booktabs}       
\usepackage{amsfonts}       
\usepackage{nicefrac}       
\usepackage{microtype}      
\usepackage{color}
\usepackage{lipsum}
\usepackage{float}
\usepackage{calrsfs}
\usepackage{enumitem}
\usepackage{tikz-cd}
\usepackage{tikz}
\usepackage{graphicx}
\usepackage{old-arrows} 
\usetikzlibrary{arrows}
\usepackage{comment}
\DeclareMathAlphabet{\pazocal}{OMS}{zplm}{m}{n}

\usepackage{eqnarray,amsmath}
\usepackage{marginnote}

\def\N{\mathbb N}

\def\span{\mathop{\hbox{\rm span}}}
\def\fl{{E^1}}
\def\h{\mathop{\hbox{\rm h}_{\hbox{\tiny\rm alg}}}}

\def\remove#1{}
\def\F{\mathcal{F}}
\def\G{\mathcal{G}}

\def\gkdim{\mathop{\hbox{\rm GKdim}}}

\xdef\Rad{2}

\newtheorem{lemma}{Lemma}[section]
\newtheorem{corollary}[lemma]{Corollary}
\newtheorem{theorem}[lemma]{Theorem}
\newtheorem{proposition}[lemma]{Proposition}
\newtheorem{remark}[lemma]{Remark}

\newtheorem*{theorem*}{Theorem}

\theoremstyle{definition}
\newtheorem{definition}[lemma]{Definition}
\newtheorem{example}[lemma]{Example}

\title[Entropy of Path Algebras]{Algebraic Entropy and a Complete Classification of Path Algebras over Finite Graphs by Growth}

\author[Bock]{Wolfgang Bock}
\address{W. Bock: Technomathematics group, Mathematics Department, Technische Universit\"at Kaiserslautern, Kaiserslautern, Germany. }
\email{bock@nathematik.uni-kl.de}

\author[Gil]{Crist\'obal Gil Canto}
\address{C. Gil Canto: Departamento de Matem\'atica Aplicada, E.T.S. Ingenier\'\i a Inform\'atica, Universidad de M\'alaga, 
M\'alaga,   Spain.}
\email{cgilc@uma.es}

\author[Mart\'in]{Dolores Mart\'in Barquero} 
\address{D. Mart\'{\i}n Barquero: Departamento de Matem\'atica Aplicada, Escuela de Ingenier\'{\i}as Industriales, Universidad de M\'alaga, 
M\'alaga, Spain.}
\email{dmartin@uma.es}

\author[Mart\'in]{C\'andido Mart\'in Gonz\'alez} 
\address{C. Mart\'{\i}n Gonz\'alez:  Departamento de \'Algebra Geometr\'{\i}a y Topolog\'{\i}a, Fa\-cultad de Ciencias, Universidad de M\'alaga, 
M\'alaga, Spain.}
\email{candido\_m@uma.es}

\author[Ruiz]{Iv\'an Ruiz Campos}
\address{I. Ruiz Campos:  Departamento de \'Algebra Geometr\'{\i}a y Topolog\'{\i}a, Fa\-cultad de Ciencias, Universidad de M\'alaga, 
M\'alaga, Spain.}
\email{ivaruicam@uma.es}

\author[Sebandal]{Alfilgen Sebandal}
\address{A. Sebandal: Department of Mathematics and Statistics, Mindanao State University - Iligan Institute of Technology, Iligan City, Philippines and Western Sydney University, Sydney, Australia.}
\email{alfilgen.sebandal@g.msuiit.edu.ph}


\begin{document}
\maketitle
 
\begin{abstract}
\textcolor{black}{The Gelfand-Kirillov dimension is a well established quantity to classify the growth of infinite dimensional algebras. In this article we introduce the algebraic entropy for path algebras. For the path algebras, Leavitt path algebras and the path algebra of the extended (double) graph, we compare the Gelfand-Kirillov dimension and the entropy. We give a complete classification of path algebras over finite graphs by dimension, Gelfand-Kirillov dimension and algebraic entropy. We show indeed how these three quantities are dependent on cycles inside the graph. Moreover we show that the algebraic entropy is conserved under Morita equivalence. In addition we give several examples of the entropy in path algebras and Leavitt path algebras. }
\end{abstract}

\section{Introduction}
In the study of infinite dimensional algebras, the notion of dimension of course has not much sense. However, instead of the dimension, one can categorize an algebra by the growth of the dimension of finitely generated subalgebras, if the algebra is filtered or graded. For commutative algebras, this results in the well-known Krull dimension, see e.g.~\cite{Eisenbud} or in the Gelfand-Kirillov dimension (GK-dimension) in the general case, see e.g.~\cite{GKdimension,SmithZhang98}. In particular, an algebra is of polynomial growth if the GK-dimension is finite and of exponential growth if it is infinite. It is somewhat unsatisfactory that a quantification of the exponential growth is not given via this concept. \textcolor{black}{For this purpose, in the case of groups and filtered algebras, the notion of algebraic entropy is introduced, see e.g.~ \cite{dikranjan2012connection,newman2000entropy}. The concept itself was sketched for the first time by Adler, Konheim and McAndrew for endomorphisms of abelian groups \cite{Adler}, or more particular endomorphisms of torsion abelian groups. This algebraic entropy, was later studied by Weiss \cite{Weiss} using  Pontryagin duality.
A different definition was given by Peters \cite{peters1979} for automorphisms of abelian groups. It coincides with the one by Weiss on endomorphisms of torsion abelian groups. The algebraic entropy was recently used and generalized in group theory in e.g.~\cite{bruno2017some,dikranjan2016entropy}. 
For graded algebras the entropy was studied in \cite{newman2000entropy}. The growth of the algebra is in this setting based on the dimensions of certain factor spaces. We will use this particular concept for filtered algebras in this article. We refer to the well-written overview article \cite{smoktunowicz2014} for more relations on growth and entropy of algebras.}
In statistical mechanics the entropy is counting the number of possible states a system can be in. The algebraic entropy can be considered similarly. The microstates that Boltzmann \cite{Boltzmann} used in his definition of the entropy are replaced by the span of finite dimensional subspaces generating the infinite dimensional algebra. The growth is then defined as the dimension of consecutive factor spaces.

In the present article we extend this concept to general path algebras and Leavitt path algebras. The Leavitt path algebra of a directed graph in the past decades received substantial interests from algebraists after its introduction in two separate, but simultaneously created, papers by Abrams and Aranda \cite{AbramsPino} and by Ara, Moreno and Pardo \cite{Aramorenopardo}. Studies look for description of the algebraic structure of these algebras as well as classifications via the geometry of the associated graphs. In \cite{Zelmanov12}, the Gelfand-Kirillov dimension (GK-dimension) of the Leavitt path algebras were investigated. In particular, it was shown that this dimension is finite only in the case where the cycles in the graph are pairwise disjoint and  a complete geometrical formula for was obtained. In \cite{Hazratsebandalvilela}, another formula of the GK-dimension was presented based on the monoid of the graded finitely generated projective modules of the Leavitt path algebra.

The generalization of the concept of entropy in this setting is the next obvious step. Indeed, we give in this article the definition of the algebraic entropy for path algebras and for Leavitt path algebras. We show that a Leavitt path algebra associated to a finite graph is either one of the three: finite dimensional, of finite GK-dimension but of entropy zero or of infinite GK-dimension but finite entropy. In addition, we give an explicit formula to obtain the entropy  for a path algebra via the adjacency matrix. In particular, we compute it for different examples with infinite GK-dimension. Furthermore, we study the behaviour of the entropy  under Morita equivalence.

\textcolor{black}{This paper is organized as follows: In Section 3, we define the algebraic entropy of an arbitrary filtered algebra in connection with the grading associated to the filtration. We also compare entropies of an algebra given different filtrations as well as the dimensions of the spaces in the filtration with the entropy and the GK-dimension. Motivated by the standard grading on the path and Leavitt path algebra, we then fix the filtrations considered in this paper. Finally, using the standard filtration, we compute the entropy of the Leavitt path algebra of a graph with a single vertex and $n$ loops as a first example. Section 3 discusses entropy of algebras under monomorphisms, epimorphisms, direct sums and Morita equivalences. In particular, bounds for the entropy of Leavitt path algebra in relation to those of the path algebra of the original graph and the extended graph were obtained. We also see that under Morita equivalence but perhaps of  different filtration, the entropy is conserved.  In Section 5, we compute the entropy of some finite graphs using the adjacency matrix, and when the graph has no cycles with a common vertex. We also classify the algebras into classes based on the dimension, GK-dimension and the entropy. Section 6, using computer algebra systems, provides computations of the entropy of some graphs based using a formula for the number of paths of certain length in the Leavitt path algebra. Lastly,  using these computations, we conjecture that the path and Leavitt path algebra of a graph have equal entropy.}

\section{Preliminaries}
In this section we give a brief definition and overview of the concepts needed from the theory of path algebras and Leavitt path algebras. Both are very intensively studied objects. Since a complete overview of all the concepts would be far too much for this article, we refer the interested reader to the articles \cite{Abramsdecade, AbramsPino, tomforde2007, abrams2008leavitt} and the monograph \cite{AAS}, as well as the references therein.

\noindent A \emph{directed graph, digraph or quiver} is a $4$-tuple $E=(E^0, E^1, s, r)$ 
consisting of two disjoint sets $E^0$, $E^1$ and two maps
 $r, s: E^1 \to E^0$. The elements of $E^0$ are called \emph{vertices} and the elements of $E^1$ are called \emph{edges} of $E$. We say that $E$ is a \emph{finite graph} if $\vert E^0 \cup E^1\vert < \infty$. For $e\in E^1$, $s(e)$ and $r(e)$ is 
 the \emph{source} and the \emph{range} of $e$, respectively. A
vertex $v$ for which $s^{-1}(v)=\emptyset$  is called a \emph{sink}, while a vertex $v$ for which $r^{-1}(v)=\emptyset$ is called a \emph{source}. 
 We will denote the set of sinks of $E$ by $\text{Sink}(E)$ and the set of sources by $\text{Source}(E)$. We say that a vertex $v \in E^0$ is a \emph{infinite emitter} if $\vert s^{-1}(v)\vert =\infty$. The \emph{set of regular
vertices} (those which are neither sinks nor infinite emitters) is denoted by $\text{Reg}(E
)$. A graph $E$ is \emph{row-finite} if  $s^{-1}(v)$ is a finite set for every $v \in E^0$. Throughout this paper we only consider finite graphs. We say that a \emph{path} $\mu$ has \emph{length} $m \in \N$, denoted by $l(\mu)=m$, if it is a finite chain of edges $\mu=e_1\ldots e_m$ such that $r(e_i)=s(e_{i+1})$ for $i=1,\ldots,m-1$.  We define $\text{Path}(E)$ as the set of all paths in $E$. We denote by $s(\mu):=s(e_1)$ the source of $\mu$ and $r(\mu):=r(e_m)$ the range of $\mu$.  We write $\mu^0$ the set of vertices of $\mu$. The vertices are the trivial paths. If $r(\mu)=s(\mu)$,
 then $\mu$ is called a \emph{closed path}. Recall that a path $\mu = e_1 \ldots e_n$, where
$e_i \in E^1$, is called a \emph{cycle} if $s(\mu) = r(\mu)$ and $s(e_i)\ne s(e_j)$ for every $i\ne j$. A cycle of length $1$ is called a \emph{loop}.  We say that $e \in E^1$ is an \emph{exit} for a cycle $\mu=e_1\ldots e_m$ if there exists an $i \in \{1,\ldots ,m\}$ such that $s(e)= s(e_i)$
and $e \ne e_i$.
Cycles $C$ and $D$ are said to be \emph{disjoint} if there is no common vertex. We will say that a cycle $C$ is  an \emph{exclusive cycle} if it is disjoint with every other cycle; equivalently, no vertex on $C$ is the base of a different cycle other than a rotated of $C$. Otherwise, we will say that $C$ is a \emph{non-exclusive cycle}. We say that $E$ satisfies \emph{Condition (EXC)} if every cycle of $E$ is an exclusive cycle.
In other words, a graph with Condition (EXC) is one without non-disjoint cycles. A $chain$ $of$ $cycles$ of $length$ $n$ is a sequence of cycles $C_1, C_2, \cdots, C_n$ such that there is a path from $C_i$ to $C_{i+1}$ for each $i<n$. This chain has an $exit$ if the last cycle $C_n$ has an exit. 

\textcolor{black}{
For a graph $E$, the  $opposite$ $graph$ $E^{\text{op}}$ of $E$ is the graph having vertices ${(E^{\text{op}})}^0:=E^0$ and edges ${(E^{\text{op}})}^1:= \{e^*:e\in E^1, s(e)=r(e^*), r(e)=s(e^*)\}$. In other words opposite graph is the obtained by taking the reverse-oriented edges of the original graph (also called the $transpose$ $graph$, see \cite{AAS}).}

\textcolor{black}{{For a ring $R(+,\cdot)$, the \emph{opposite ring} $R^{\text{op}}(+,\times)$ is the ring in which the multiplication $\times$ is defined by $a\times b = b \cdot a$, i.e.~ it is performed in reverse order. The \emph{opposite algebra} is defined in the same way.}}

 Given a directed graph, one can associate an algebra which respects the geometry of the graph, in some way. 
 
 \indent For a graph $E$ and a ring $R$ with identity, the  \textit{Path algebra} of $E$, denoted by $RE$, is the $R$-algebra generated by the sets $\{v:v\in E^0\}$ and $\{e \colon e \in E^1  \}$ with coefficients in $R$, subject to the relations:
\begin{enumerate}
\item[\textnormal{(V)}]
$v_iv_j=\delta_{i,j}v_i$ for every $v_i, v_j\in E^0$;
\item[\textnormal{(E)}] $s(e)e=e = e r(e)$ for all non-sinks $e \in E^1$.
\end{enumerate}

 The \emph{extended graph $E$} is defined as the new graph $\widehat{E}=(E^0, E^1\cup (E^1)^*, r',s'  )$, where $(E^1)^*= \{ e^*:e\in E^1 \}$ and the maps $r'$ and $s'$ are defined as
$r'|_{E^1}=r$, $s'|_{E^1}=s$, $r'(e^*)=s(e)$, and $s'(e^*)=r(e)$ for all $e\in E^1$. In other words, each $e^*\in (E^1)^*$ has orientation the reverse of that of its counterpart $e\in E^1$. The elements $(E^1)^*$ are called $ghost$ $edges$.

 \indent For a graph $E$ and a ring $R$ with identity, the  \textit{Leavitt path algebra} of $E$, denoted by $L_R(E)$, is the path algebra over the extended graph $\hat{E}$ with additional relations
\begin{enumerate}
\item[\textnormal{(CK1)}] $e^*e'=\delta_{e,e'}r(e)$ for all $e, e'\in E^1$;

\item[\textnormal{(CK2)}]$\sum_{ \{e \in E^1 \colon s(e)=v    \}  } e e^*=v$ for every $v\in \text{Reg}(E)$.

\end{enumerate}
Throughout this paper, we are working with path algebras with coefficients over a field $K$.

Two of fundamental examples of path algebras and Leavitt path algebras over a graph are the so-called $1$-petal rose $R_1$ and the oriented $n$-line graph $A_n$.

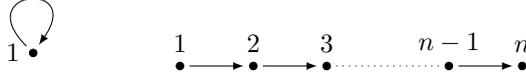
\begin{figure}[ht]\label{1petalrose-nline}
    
         \begin{tikzpicture}[scale = 0.65, shorten <=2pt,shorten >=2pt,>=latex, node distance={5mm},sub/.style = {draw, fill, circle, inner sep = 1pt}, main/.style = {draw, fill, circle, inner sep = 1pt}, sub/.style = {draw = red, fill = red, circle, inner sep = 1pt}]
    \node[main,label = left:$1$] (1) at (3,0) {};
    
    \draw[->] (1) to  [out = 135, in = 55, looseness = 40] node[auto] {} (1);
    
    \end{tikzpicture} \ \ \ \ \ \ \ \ \begin{tikzpicture}[scale = 0.65, shorten <=2pt,shorten >=2pt,>=latex, node distance={5mm},sub/.style = {draw, fill, circle, inner sep = 1pt}, main/.style = {draw, fill, circle, inner sep = 1pt}, sub/.style = {draw = red, fill = red, circle, inner sep = 1pt}]
    
    \node[main,label = above:$1$] (1) at (3,0) {};
     \node[main,label = above:$2$] (2) at (4.5,0) {};
     \node[main,label = above:$3$] (3) at (6,0) {};
      \node[main,label = above:$n-1$] (4) at (8.5,0) {};
      \node[main,label = above:$n$] (5) at (10,0) {};
    \draw[->] (1) to  (2) {} ;
    \draw[->] (2) to  (3) {} ;
    \draw[dotted] (3) to (4) {} ;
    \draw[->] (4) to  (5) {}; 
    \end{tikzpicture}
    \caption{The graphs $R_1$ and $A_n$.}
   
\end{figure}

The path algebras of these graphs are $KR_1=K[x]$, the polynomial algebra with coefficients in $K$ and $KA_n=T_n(K)$, the upper triangular matrix algebra over $K$, respectively. On the other hand, the Leavitt path algebras are $L_K(R_1)=K[x,x^{-1}]$, the Laurent polynomial algebra and $L_K(A_n)=M_n(K)$, matrix algebra over $K$, respectively. Other basic definitions and results on graphs and Leavitt path algebras can be seen in the book \cite{AAS}.

We will apply algebraic entropy ideas in this work both to path and to Leavitt path algebras. Since the definition of algebraic entropy is related to the one of Gelfand-Kirillov dimension, we start by recalling the latter.

\begin{definition}
\cite{GKdimension}  Given functions $f,g:\mathbb{N} \rightarrow \mathbb{R}^+$, we will use the notation $f\preccurlyeq g$ if there exists $c\in \mathbb{N} $ such that $f(n)\leq cg(cn)$ for all $n\in \mathbb{N} $.
 If $f\preccurlyeq g$ and $g\preccurlyeq f$, the functions $f$ and $g$ are said to be $asymptotically$ $ equivalent$ denoted by $f\sim g$. In this case we have 
 $$\lim_{n \to \infty} \displaystyle \frac{\log(f(n))}{\log (n)}= \lim_{n \to \infty} \displaystyle \frac{\log(g(n))}{\log (n)}.$$ Hence, the limit does depend only on the equivalence class of $f$ under $\sim$. This equivalence class of $f$ is called the {\em growth} of $f$.\\
\end{definition}

\begin{definition}
\cite{GKdimension}  Let $A$ be an algebra, which is generated by a finite dimensional subspace $V$. Let $V^n$ denote the span of all products $v_1v_2\cdots v_k,$
$v_i\in V$, $k\leq n$. Then $V=V^1\subseteq V^2\subseteq \cdots$,
\begin{equation*}
   A=\bigcup_{n\geq 1}V^n \ \ \text{and } \ \ g_V(n):=\dim V^n<\infty.
\end{equation*}
If $W$ is another finite-dimensional subspace that generates $A$, then $g_V(n)\sim g_{W}(n)$. If $g_V(n)$ is polynomially bounded, then  the {\em Gelfand-Kirillov} dimension of $A$ is defined as 
\begin{equation*}
    \gkdim (A) := \displaystyle \limsup_{n\rightarrow \infty } \frac{\log g_V(n)}{ \log(n)}.
\end{equation*}
The GK-dimension does not depend on a choice of the generating space $V$ as long as $\dim (V)<\infty$. If the growth of $A$ is not polynomially bounded, then $\gkdim (A)=\infty$.
\end{definition}

Note that if $g_V(n) = n^k$, then $\gkdim (A) = k$. In the case of a finite dimensional algebra, we have that the Gelfand-Kirillov dimension is zero.

In particular, for Leavitt path algebras there is a recent work which classifies all Leavitt path algebras having Gelfand-Kirillov dimension less than $4$, see \cite{Koc}. Moreover, a well-known result linking the Gelfand-Kirillov dimension of a Leavitt path algebra to graphs with the condition (EXC) was formulated. This computation is based on the length of certain chain of cycles on the graph. 
\begin{theorem}\textnormal{\cite{Zelmanov12}}\label{zelgk} 
Let $E$ be a finite  graph.
\begin{enumerate}
\item [\textnormal{(1)}] The Leavitt path algebra $L_K(E)$ has polynomially bounded growth if and only if $E$ satisfies (EXC) condition.
\smallskip
\item [\textnormal{(2)}] If $d_1$ is the maximal length of a chain of cycles in $E$, and $d_2$ is the maximal length of chain of cycles with an exit, then $$\gkdim(L_K(E)) = \max(2d_1-1, 2d_2).$$
\end{enumerate}
\end{theorem}
This theorem was also used to classify Lie bracket algebras over Leavitt path algebras in \cite{BSV22}.

The Gelfand-Kirillov dimension indeed gives us the exponent in the polynomial growth of an algebra. If we have in an algebra $A$ that $g_V(n)\preccurlyeq n^k$ then indeed $\gkdim(A)=k$. \textcolor{black}{In this paper, we will see that the (EXC) condition also plays a role in the entropy.}


\section{Entropy for filtered algebras.}
Whereas the Gellfand-Kirillov dimension is independent of the finite dimensional system of generators chosen, the concept of entropy is more subtle. In particular it only makes sense if there is what we will call a filtration of the algebra. 

A $K$-algebra $A$ is said to be \emph{filtered} if it is endowed with a collection of subspaces $\mathcal{F}=\{V_n\}_{n=0}^\infty$ such that
\begin{enumerate}
\item $0=V_0\subset V_1\subset\cdots\subset V_n\subset V_{n+1}\subset\cdots A$,  
\item $A=\cup_{n\ge 0}V_n$,
\item $V_nV_m\subset V_{n+m}$.
\end{enumerate}

We consider the \emph{category of filtered} $K$-algebras. Its objects are the couples $(A,\F)$ where $\F$ is a filtration on $A$ and its morphisms $(A,\F)$ to $(B,\G)$ are the $K$-algebra morphisms $f\colon A\to B$ such that $f(V_i)\subset W_i$, for all $V_i \in \F$ and $W_i \in \G$ . Note that an isomorphism $f\colon (A,\F)\to (B,\G)$ implies $\dim(V_i)=\dim(W_i)$ for any $i$. Moreover, if the filtrations consist of finite-dimensional subspaces, then 
$f(V_i)=W_i$ for any $i$.

For a graded algebra $A=\oplus_{i=0}^\infty A_i$, its entropy has been defined in Equation  \cite[(1), p. 85]{newman2000entropy} as 
$$H(A)=\limsup_{n\to\infty}\root n \of{\dim (A_n)}.$$
 Now, we are defining the algebraic entropy of a filtered algebra $(A,\F)$, essentially as the logarithm of $H(A)$ of the graded algebra associated to the filtration. 
To be more precise, we choose this particular definition in order to have a certain form of comparability with the Gelfand-Kirillov dimension. If $(A,\F)$ is an algebra with a filtration $\F=\{V_n\}_{n\ge 0}$ of finite-dimensional quotients $V_n/V_{n-1}$, then we can consider the associated graded algebra 
$\text{\bf gr}(A):=\oplus_{i\ge 0}V_{i+1}/V_i$ with product $(x+V_{n-1})(y+V_{m-1}):=xy+V_{n+m-1}$
where $x+V_{n-1}\in V_{n}/V_{n-1}$, $y+V_{m-1}\in V_{m}/V_{m-1}$. Hence, we define the \emph{algebraic entropy of a filtered algebra} $(A,\F)$, $$\h(A,\F):=
\begin{cases}
0 &\text{ if } $A$ \text{ is finite dimensional,} \\
\displaystyle \limsup_{n\to\infty}\frac{
\log\dim(V_n/V_{n-1})}{n} & \text{ otherwise.}
\end{cases}$$

\begin{remark}\rm
The only care we should have with this definition is that in case there is some step in the filtration for which $V_{n-1}=V_n$, then we have a $-\infty$ in the sequence. So the sequence $n^{-1}\log\dim(V_n/V_{n-1})$ takes values in $\mathbb{R}\cup\{-\infty\}$. Thus the above limit could be $-\infty$ if
the filtration stabilizes at some point. But this is only possible when $A$ is finite-dimensional and we have defined $\h(A,\F)=0$ in this case.
\end{remark}

Observe also that $\h(A, \F)=\h(A^{\text{op}},\F)$ where $A^\text{op}$ is the opposite algebra of $A$. In particular, since $(KE)^\text{op}=K E^{\text{op}}$ being $E^\text{op}$ the opposite graph, we have $\h(KE,\F)=\h(KE^\text{op},\F)$.

If there is no doubt about the filtration $\F$ that we are considering in $A$, then we can shorten the notation 
$\h(A,\F)$ to $\h(A)$. However, we have to be careful because the entropy of a filtered algebra depends on the chosen filtration. This can be illustrated in the next proposition.

\begin{proposition}\label{huevosconpatatas}
Let $\F=\{V_n\}$ be a filtration of a finitely generated algebra $A$. For the filtration $\G=\{W_n\}$ such that $W_n:= V_{nk}$ for any $n\in \N$ and a fixed $k\in \N^*$, one has $\h(A,\G)=k \cdot \h(A,\F)$.
\end{proposition}

\begin{proof} First, for every $i$, we have that $\displaystyle \dim(V_{ki}/V_{ki-k}) = \sum_{j=0}^{k-1} \dim(V_{ki-j}/V_{ki-j-1})$. Then, 
\begin{eqnarray*}
    \h(A,\mathcal{G})&=&  \limsup_{i \to \infty}\frac{\log\left (\dim(V_{ki}/V_{ki-k}) \right )}{i}\\
   &=& \limsup_{i \to \infty} \frac{1}{i}\log \left (\displaystyle\sum_{j=0}^{k-1}\dim(V_{ki-j}/V_{ki-j-1})\right ).
\end{eqnarray*}

Secondly, take into account that for two sequences $\{a_n\}_{n\geq 0}$ and $\{b_n\}_{n\geq 0}$ with $a_n,\ b_n>0$ the following equality holds: $$\limsup_{n \to \infty} \frac{\log(a_n +b_n)}{n}=\max \left \{\limsup_{n \to \infty} \frac{\log(a_n)}{n}, \limsup_{n \to \infty}\frac{\log(b_n)}{n}\right \}.$$
Consequently, we have

\begin{eqnarray*}
\h(A,\mathcal{G})&=&\max_{0\leq j\leq k-1}\left \{\limsup_{i \to \infty} \frac{ki-j}{i} \frac{\log(\dim(V_{ki-j}/V_{ki-j-1}))}{ki-j} \right \}\\
&=& k \cdot \max_{0 \leq j \leq k-1}\left \{\limsup_{i \to \infty} \frac{\log(\dim(V_{ki-j}/V_{ki-j-1}))}{ki-j} \right \}\\
&=& k \cdot \h(A,\F).
\end{eqnarray*}

\end{proof}
\begin{remark}\rm
    Observe that Proposition \ref{huevosconpatatas} implies that it is not possible to obtain a well-defined entropy (similar to the definition in \cite{dikranjan2012connection}) via the supremum since if we have a filtration which leads us to a nonzero value of the entropy, we obtain a sequence of filtrations with increasing entropy. 
\end{remark}

If there is a filtration for which the entropy is finite and nonzero, we can obtain a family of filtrations with increasing entropy. However, if the value of the entropy of this first filtration is zero, the entropy remains constant (equal to zero) for the families of filtrations mentioned below. The following result may suggest that if the entropy for a given filtration is zero, then  the entropy is zero for any filtration.

\begin{proposition}\label{sopadecebolla} Assume that $A$ is a finitely generated algebra and $\h(A,\F)=0$ for a filtration $\F=\{V_n\}$. Then we have the following.
\begin{enumerate}
    \item For any other filtration $\G=\{W_n\}$ such that $W_n\subset V_n$ for any $n$, one has $\h(A,\G)=0$.
    \item For the filtration $\G=\{W_n\}$ such that $W_n:= V_{nk}$ for any $n$ and a fixed $k$, one has $\h(A,\G)=0$.
    \item \label{mandarina} For any other filtration $\G=\{W_n\}$ such that $W_1$ is finite dimensional and 
    $W_k=(W_1)^k$ (for any $k$), one has $\h(A,\G)=0$.
\end{enumerate}
\end{proposition}
\begin{proof}
 Since we have $\limsup_{n\to\infty}\frac{\log \Delta\dim(V_n)}{n}=0$ (where $\Delta x_n=x_n-x_{n-1}$ for $n>1$), we have $\lim_{n\to\infty}\frac{\log\Delta\dim(V_n)}{n}=0$ and for any
$\varepsilon>0$ there is some $n_0$ such that 
\begin{equation}\label{ladrido}
e^{-\varepsilon n}< \Delta\dim(V_n)<e^{\varepsilon n}, 
\end{equation}
when $n>n_0$. Then $$e^{-\varepsilon(n_0+1)}+\cdots+e^{-\varepsilon n}<\sum_{i=n_0+1}^n\Delta\dim(V_i)<e^{\varepsilon(n_0+1)}+\cdots+e^{\varepsilon n}.$$
Thus, there is a constant $k$ such that
 $$k+\sum_{i=n_0+1}^ne^{-\varepsilon i}<\dim(V_n)<k+\sum_{i=n_0+1}^ne^{\varepsilon i}.$$
So $\dim(W_n)< k+\sum_{i=n_0+1}^ne^{\varepsilon i}$ implying 
$\Delta\dim(W_n)< k+\sum_{i=n_0+1}^ne^{\varepsilon i}$.
But $\sum_{i=n_0+1}^ne^{\varepsilon i}+k=He^{\varepsilon n}+k'$ for some constants $k',H$. In consequence, we have that $\log \left (\Delta\dim(W_n)\right )<
\log\left(H e^{n\varepsilon}+k'\right)$. Finally, $\h(A,\G)\le \displaystyle \lim_{n\to\infty}
\frac{\log\left(H e^{ n\varepsilon}+k'\right)}{n}=\varepsilon$.

The second item is a direct consequence of Proposition \ref{huevosconpatatas}.

In order to prove the third item, we know that $W_1\subset V_k$ for some $k$ that we fix and define next the new filtration $\mathcal{H}:=\{Z_n\}$ where $Z_n:=V_{nk}$ for any $n$. By the previous item (2), $\h(A,\mathcal{H})=\h(A,\F)=0$. On the other hand, for any $n$ we have 
$W_n=(W_1)^n\subset V_k^n\subset V_{nk}=Z_n$ hence again by the previously proved item (1) we have $\h(A,\G)=\h(A,\mathcal{H})=0$.
\end{proof}

\begin{definition}\rm  Let $\{a_n\}$ and $\{b_n\}$ be two sequences of positive real numbers. We say that $\{a_n\}$ has the same \emph{order} of $\{b_n\}$ and we denote it by $\Theta(a_n)=\Theta(b_n)$ if there exists $c \in {\mathbb R}^{\times}$ such that $\displaystyle \lim_{n\to \infty} \frac{a_n}{c \cdot b_n}=1$. Also we will write $\Theta(a_n) > \Theta(b_n)$ if $\displaystyle \lim_{n \to \infty} \frac{a_n}{b_n}=\infty$.
\end{definition}

\begin{proposition} \label{ducha}
Assume $a_n, b_n$ are positive  sequences, with $b_n$ strictly increasing and $\lim_{n\to\infty}b_n=\infty$. Furthermore, let $f$ be a strictly monotone increasing function with $\lim_{n\rightarrow \infty }f(n)=\infty$. We will use the 
notation $\Delta a_n:=a_n-a_{n-1}$ and similarly for $b_n$, $(n>1)$. If $\Theta(a_n)<\Theta(b_n)$, then we have
\begin{enumerate}
    \item $\Theta(\Delta a_n)<\Theta(\Delta b_n)$,
    \item $\displaystyle \limsup_{n\to\infty}\frac{\log(a_n)}{f(n)}\le \limsup_{n\to\infty}\frac{\log(b_n)}{f(n)}$,  
    \item  $\displaystyle \limsup_{n\to\infty}
\frac{\log(\Delta a_n)}{f(n)}\le \limsup_{n\to\infty}\frac{\log(\Delta b_n)}{f(n)}$. 
\end{enumerate}
\end{proposition}
\begin{proof} For the first assertion we use the known formula $\displaystyle \limsup\frac{a_n}{b_n}\le\limsup\frac{\Delta a_n}{\Delta b_n}$.
For the second item we know that for any positive $M$ we have $b_n>M a_n$ for large enough $n$. Finally, the last assertion is a consequence of the first one because this means that for any positive $N$ we have that $\Delta b_n > N \Delta a_n$ for large enough $n$. 
\end{proof}

\begin{lemma}\label{manteca}
Let $A$ be a $K$-algebra and $\F = \{V_n\}$ a such that $V_1$ is a finite dimensional system of generators and verifying $(V_1)^n = V_n$. Suppose that the sequence $\{\dim(V_n/V_{n-1})\}$ is bounded, then $\gkdim(A)\le 1$. Furthermore, if $\{\dim(V_n/V_{n-1})\}$ is convergent, then $\gkdim(A)=1$.
\end{lemma}
\begin{proof}
If we consider $\dim(V_n/V_{n-1}) < M$ for all $n$, then $$0\le \sum_{k=2}^n \left [\dim(V_k)-\dim(V_{k-1}) \right ] < (n-1)M.$$ If we put $k:=\dim(V_1)$, then 
$0\le \dim(V_n)-k < (n-1)M$ or equivalently 
$$k\le \dim(V_n) < (n-1)M+k,$$
$$\log k\le \log \dim(V_n) < \log[(n-1)M+k],$$
$$\frac{\log k}{\log n}\le \frac{\log \dim(V_n)}{\log n} < \frac{\log((n-1)M+k)}{\log n}$$
and taking limits we get that $\gkdim(A)\le 1$.  

If the sequence $\{\dim(V_n/V_{n-1})\}$ is not only bounded but convergent, we can go a little further. If $\lim_{n\to \infty} \dim \left (V_n/V_{n-1} \right )= c$, taking into account that  $\dim  \left (V_n \right )$ is a sequence of natural numbers, we have that $\dim \left (V_n/V_{n-1} \right ) = c$ for every $n\geq k_0$ with $k_0 \in \mathbb{N}$. Thus, $\dim (V_n )- \dim (V_{n-1}) = c$ and $\dim ( V^n) = \dim (V_n) = \dim (V_{n-1}) + c$ for every $n\geq k_0$. Finally, $\dim (V^n) = \dim (V_{k_0}) + c(n-k_0)$ and we can compute
\begin{equation*}
    \gkdim(A) = \limsup_{n \to \infty} \frac{\log \left (\dim ( V^n) \right )}{\log n} = \limsup_{n\to \infty} \frac{\log \left ( \dim (V_{k_0})+ c(n-k_0) \right )}{\log n} = 1. 
\end{equation*}
\end{proof}

\begin{proposition}
Suppose that $A$ is a $K$-algebra and $\F = \{V_n\}$ a filtration of $A$ with $V_1$ a finite dimensional system of generators with $(V_1)^n = V_n$. 
\begin{enumerate}
    \item If $\lim\limits_{n\to \infty} \dim(V_n/V_{n-1}) = 0$, then 
    $h_{alg} (A)=0$ 
   and 
    $
    \mathrm{GKdim}(A) = 0.
    $
    \item If $\lim\limits_{n\to \infty} \dim(V_n/V_{n-1}) =  c>0$, then 
    $h_{alg} (A)=0$ 
   and 
    $
    \mathrm{GKdim}(A) = 1.
    $
    \item If $\dim(V_n) =  \Theta(n^k)$  for some $k \in \mathbb{N}^*$, then 
    $
    h_{alg} (A) =0,
    $
    and 
    $
     \mathrm{GKdim}(A) =k$.
  \item   If $\dim(V_n)=\Theta(a^n)$, then $\h(A)=\log(a)$ and $\gkdim(A)=\infty$.
\end{enumerate}
\end{proposition}
\begin{proof}
    \begin{enumerate}
        \item If $ \lim_{n\to \infty}\dim \left (V_n/V_{n-1} \right )= 0$, then $V^n = V^{n-1}$ for every $n\geq k_0$. This means that $A = \cup_{n=1}^\infty V_n = V_{k_0}$ and $A$ is finite dimensional. Then $\h(A) = \gkdim(A) = 0$. 
        \item It is a consequence of  Lemma \ref{manteca}.
        \item If $\dim(V_n)=\Theta(n^k)$, then there is a constant $c$ such that $\displaystyle \lim_{n\to\infty}\frac{\dim(V^n)}{c n^k}=1 $. Thus $1-\varepsilon <\displaystyle  \frac{\dim(V_n)}{cn^k}<1+\varepsilon$ for large enough $n$. So $(1-\varepsilon)c n^k< \dim(V_n)<(1+\varepsilon)c n^k$. Then
   $\log[(1-\varepsilon)c n^k]< \log \dim(V_n)<\log[(1+\varepsilon)c n^k]$ and
   $\displaystyle \frac{\log[(1-\varepsilon)c n^k]}{\log n}< \frac{\log \dim(V_n)}{\log n}<\frac{\log[(1+\varepsilon)c n^k]}{\log n}$
   and taking limits we get $\gkdim(A)=k$. Next we compute the entropy. We have seen that (for large enough $n$) we have
   $(1-\varepsilon)c n^k< \dim(V_n)<(1+\varepsilon)c n^k$ and taking $\varepsilon=1$ we have
   $$\dim(V_n)<2 c n^k$$
   and consequently $$\dim(V_n/V_{n-1})<2cn^k$$
   $$\log \dim(V_n/V_{n-1})<\log 2cn^k$$
   $$\frac{\log \dim(V_n/V_{n-1})}{n}<\frac{\log 2cn^k}{n}$$ and taking limits we get 
   $\h(A)=0$.
   \item  Suppose $\dim(V_n) =  \Theta(a^n)$, for some $a \in \mathbb{R}$. It is known  that if $\{x_n\},\{y_n\}$ are two sequences such that $\lim_{n \to \infty} y_n=\infty$ and $\lim_{n \to \infty} \frac{x_n}{y_n}=1$, then $\lim_{n \to \infty}\frac{\log(x_n)}{\log(y_n)}=1$. So, assume that $\lim_{n \to \infty} \frac{\dim(V_n)}{ca^n}=1$ for a certain $c\ne 0$. Then 
    $$
     \mathrm{GKdim}A = \limsup_{n \to \infty}\frac{\log(\dim(V_n))}{\log n}= \limsup_{n \to \infty}\frac{\log(ca^n)}{\log n}=
     \infty,
    $$
    and
    $$
    h_{alg} (A) = \limsup_{n \to \infty} \frac{\mathrm{log}(ca^n)}{n} =\mathrm{log}(a).
    $$
    \end{enumerate}
\end{proof}

One of the most intensively used properties of a Leavitt path algebra, but also of path algebras themselves is that $L_K(E)$, $KE$ and $K\hat E$ are $\mathbb{Z}$-graded $K$-algebras. Many structural results of path algebras are based on their grading. For the particular case of Leavitt path algebras we refer to \cite{AAS}, but also the works on talented monoids and graded ideals \cite{Hazrat}. 
This grading will motivate the standard filtration for the path algebra $KE$ and for the Leavitt path algebra $L_K(E)$.

\begin{definition}\label{churros} \rm For $KE$ we define the filtration $\{V_i\}_{i\in\N}$ where $V_0$ is the linear span of the set of vertices of the graph $E$, while
$V_1$ is the sum of $V_0$ with the linear span of the set of edges, and $V_{k+1}$ linear span of the set of paths of length less or equal to $k+1$. We will call this the {\em standard filtration on $KE$}. 
\end{definition}
\begin{definition}\rm 
\label{teje}
For $L_K(E)$ we define its {\it standard filtration} $\{W_i\}_{i \in \mathbb{N}}$  so that $W_0$ is  the linear span of the set of vertices of $E$, being $W_1$ the sum of $W_0$ plus the linear span of the set
$E^1\cup (E^1)^*$. For $W_{k}$ we take the linear span of the set of elements:  $\lambda\mu^*$ with $l(\lambda)+l(\mu)\le k$. 
\end{definition} 

From now on, any path algebra $KE$ will be understood to be endowed with its standard filtration by default (and the same applies to $L_K(E)$).

\begin{remark}\rm
For a graph with finite-dimensional path algebra, it is easy to see that we obtain entropy $0$. This is the case in particular for a finite acyclic graph $E$, since for some $m$, $E$ will have zero paths of length greater or equal $ m$. 
\end{remark}

In order to compute the Gelfand-Kirillov dimension of a path or of a Leavitt path algebra $A$, we can consider their corresponding standard filtration, say $\{V_n\}_{n\ge 0}$. In each case, the space $V_1$ is a finite-dimensional system of generators. So, in order to compute $\gkdim(A)$ we can take $W=V_1$ so that $W^n=V_n$ and $\displaystyle \gkdim(A)=\lim_{n\to\infty} \frac{\log(\dim (W^n))}{\log n}$. The equality $W^n = V_n$ for $n\ge 2$ comes from the definition of the standard filtration. The space generated by the paths of length less than or equal to $n$ coincides with the space generated by the products of less than or equal to $n$ generators of the algebra.

\medskip

To end this section, we consider the $n$-petals rose graph $R_n$, that is, one vertex and $n$ loops. We can observe how the finite entropy of $L_K(R_n)$ is depending of $n$ while the Gelfand-Kirillov dimension (for $n \geq 2$) is always infinite. This is a first example that shows how the entropy allows us to differentiate algebras with the same Gelfand-Kirillov dimension.

\begin{example}\label{cena_navidad}\rm
Take the directed graph $E=R_n$ of the $n$-petals rose graph with one vertex $v$ and $n$ loops $f_1,\ldots,f_n$ (so $s(f_i)=r(f_i)=v$ for any $i$). Consider its standard filtration.

Note that 

$\dim(V_1)-\dim(V_0)=2n$,

$\dim(V_2)-\dim(V_1)=3n^2-1$,

$\dim(V_3)-\dim(V_2)=4n^3-2n$, and proceeding in this way.

A system of generators of $V_k$ is that of 
$V_{k-1}$ plus the elements of $(\fl)^k\cup(\fl)^{k*}\cup\left(\cup_{i+j=k} (\fl)^i(\fl^{j*})\right)$. To get a basis we must remove from each $(E^1)^i(E^{1*})^j$ the elements $(E^1)^{i-1}f_1f_1^*(E^{1*})^{j-1}$ (so remove $n^{i+j-2}=n^{k-2}$ elements) (apply \cite[Corollary 1.5.12]{AAS}). 
This gives $\dim(V_k)-\dim(V_{k-1})=n^k+n^k+(k-1)(n^k-n^{k-2})=
(k+1)n^k-(k-1)n^{k-2}$

Since $\dim(V_k/V_{k-1})=(k+1)n^k-(k-1)n^{k-2}$ we have 
$$\h(L_k(R_n))=\limsup_{k\to\infty}\frac{\log[(k+1)n^k-(k-1)n^{k-2}]}{k}$$
but $$\lim_{k\to\infty}\frac{\log[(k+1)n^k-(k-1)n^{k-2}]}{k}=
\lim_{k\to\infty}\log\left[\frac{(k+1)n^k-(k-1)n^{k-2}}{kn^{k-1}-(k-2)n^{k-3}}\right]=$$
$$\lim_{k\to\infty}\log\left[\frac{n^{k-2}[(k+1)n^2-(k-1)]}{n^{k-3}[k n^2-(k-2)]} \right]=\lim_{k\to\infty}\log \left[n\ \frac{(k+1)n^2-(k-1)}{k n^2-(k-2)} \right]=$$
$$\lim_{k\to\infty}\log \left[n\ \frac{k(n^2-1)+n^2+1}{k (n^2-1)+2)} \right]=\lim_{k\to\infty}\log \left[n\ \frac{(n^2-1)+\frac{n^2+1}{k}}{ (n^2-1)+\frac{2}{k})} \right]=\log(n).$$
In conclusion we have that $L_K(R_n)=\log(n)$.
\end{example}

\section{Behaviour of the entropy relative to different constructions. Morita equivalence. }
In this section we study how the entropy of an algebra behaves under epimorphisms, monomorphisms, direct sums and Morita equivalence. If we have algebras related by an epimorphism or a monomorphism, their corresponding entropies will be related by an inequality. Thus, we can give boundaries for the entropy of the Leavitt path algebra $L_K(E)$ depending on the entropy of the corresponding path algebras $KE$ and $K\hat{E}$. The relation between entropy and direct sums will be useful because since it will allow us to compute the entropy of path algebras and Leavitt path algebras associated to disconnected graphs. 

\subsection{Epimorphisms and entropy}
Consider an epimorphism of algebras $f\colon A\to B$. If $\F$ is a filtration on $A$, then we can construct a filtration $\G$ on $B$ simply applying $f$ to the subspaces of $\F$. So we have an epimorphism in the category of filtered algebras which we denote as 
$f\colon (A,\F)\to (B,\G)$. If there is no danger of confusion, we use the same symbol $f$ for both the epimorphism of algebras and of filtered algebras.

As the epimorphisms contract dimensions, we have  $$\dim \left ( f(V_i)/f(V_{i-1}) \right )\le \dim \left ( V_i/V_{i-1} \right ). $$ 

Observe that the canonical epimorphism $p\colon K\hat E\to L_K(E)$ satisfies $p(V_i)=W_i$, as can be proved, for instance, by using induction on $i \in \N$. Consequently, we have the following Lemma.

\begin{lemma} 
For the  epimorphism of filtered algebras $f\colon (A,\F)\to (B,\G)$ constructed above, one has $\h(B)\le \h(A)$.
  In particular, for a finite graph $E$
\begin{equation}\label{acot}
\h(L_K(E))\le \h(K\hat E).
\end{equation}
\end{lemma}

In order to relate the algebraic entropy of $KE$ with that of $L_K(E)$ we use the canonical monomorphism $j\colon KE\to L_K(E)$ mapping any vertex to itself (its image in the Leavitt path algebra) and the same with any edge. We identify $KE$ with its image in $L_K(E)$.
Next we prove that $W_i \cap KE= V_i$, by induction on $i$. The equality is clear for $i=0$. Assume that $W_i \cap KE=V_i$ and let us prove $W_{i+1} \cap KE=V_{i+1}$. Take $\lambda \mu^* \in (W_{i+1}\setminus W_i) \cap KE$ such that $l(\lambda) + l(\mu)=i+1$. Since $\lambda \mu^* \in KE$, $\lambda=f_1\ldots f_n$ and $\mu=g_1 \ldots g_k$, that is, $\lambda \mu^*=f_1 \ldots f_{n-k}$ implying that $\lambda \mu^* \in V_{i+1}$ (observe that by hypothesis $n+k=i+1$ so then $n-k \le i+1$). The other containment is straightforward.

\subsection{Monomorphisms and entropy}
Now, consider a subalgebra $B$ of an algebra $A$ and the inclusion monomorphism $i \colon B \to A$. If $\F=\{V_n\}_{n \geq 0}$ is a filtration on $A$, then we can construct a filtration $\G = \{W_n\}_{n \ge 0}$ on $B$ given by $W_n = V_n \cap B$. 
Then, the inclusion $i\colon B\to A$ is a monomorphism in the category of filtered algebras.  
Since we have that,
$$\dim \left ( W_i/W_{i-1} \right )\le \dim \left ( V_i/V_{i-1} \right ),$$
the following lemma holds. 
\begin{lemma}\label{cota_monomorfismo} For the monomorphism of filtered algebras
 $i \colon (B,\mathcal{G}) \to (A,\F)$ as above, we have that $\h(B) \le \h(A)$.

Thus, for a finite graph $E$ we have the bound
\begin{equation} \label{actot2}
    \h(KE)\le\h(L_K(E)).
\end{equation}
In particular, one obtains together with the previous observations
\begin{equation}\label{eqone}
\h(KE)\le \h(L_K(E))\le \h(K\hat E).\end{equation}
\end{lemma}
\begin{proof}
  To prove this, we just need to consider Equations \eqref{acot} and \eqref{actot2}.
\end{proof}

\begin{remark}\rm
By the definition of the entropy and the Gelfand-Kirillov dimension, it is easy to see that for a directed graph $E$,
$$
\h (K(E)) \leq  \mathrm{GKdim} (K(E)),
$$
where both sides can be $\infty$ and the filtration considered is the standard one.

\end{remark}
\begin{remark}\rm  Let $A$ be finitely generated $K$-algebra with no unit and $A_1$ its unitization. If $A\ne A_1$, for any filtration $\{V_i\}_{i\ge 0}$ on $A$, we have an induced filtration $\{K\times V_i\}_{i\ge 0}$ on $A_1=K\times A$. Relative to these filtrations, it is easy to prove that $\h(A)=\h(A_1)$ by Lemma \ref{cota_monomorfismo}.
\end{remark}
Next, we will prove some bounds for algebraic entropy following its definition \eqref{churros}. In forthcoming sections, we will show that the computation of the algebraic entropy can be greatly relieved by using certain techniques using norms and spectral radius. But for now, we will be content to compute it from its very definition.

For instance, a universal bound for $\h(KE)$ when $E$ is finite is $\log(\vert E^1\vert)$:
 
\begin{proposition}
If $E$ is a finite directed graph with $\vert E^1\vert=n$,  then $\h(K E)\le\log(n).$
\end{proposition}
\begin{proof}
Assume that $\vert E^1\vert=n$ with $n > 1$. Consider the standard filtration as in Definition \ref{churros}.

So we have $\dim \left (V_1/V_0 \right )=n$ and since $V_2=V_1\oplus\span\{f_if_j\colon f_i,f_j\in E^1\}$, then $\dim \left (V_2/V_1 \right )\le n^2$. In general $\dim \left ( V_k/V_{k-1} \right )\le n^k$.
Consequently, $$\h(KE)\le \limsup_{k\to\infty}\frac{\log(n^k)}{k}=\log(n).$$
\end{proof}
\begin{corollary}
    If $E$ is a finite directed graph with $|E^1| = n$, then $\h(L_K(E)) \leq \log(2n)$.
\end{corollary}

Note that for the $n$-petal rose $R_n$, the entropy of $K R_n$ is precisely $\log n$ because the bound given for $\dim \left (V_k/V_{k-1} \right )$ in the proof above is really an equality, that is, 
$$\h(K R_n)=\log n.$$
Applying now formula \eqref{eqone} we have $\log n\le \h(L_K(R_n))\le \log(2n)$. 
Since the amplitude of the interval $[\log n,\log 2n]$ is $\log 2=0.69$ we have
a certain control on the entropy of the Leavitt path algebra $L_K(R_n)$. 

On the other hand, since $\displaystyle \lim_{n\to\infty}\frac{\log(2n)}{\log n}=1$, we claim that asymptotically the entropy of $KR_n$
and that of $L_K(R_n)$ agree, more precisely 
$$\lim_{n\to\infty} \frac{\h(L_K(R_n))}{\h(K R_n)}=1.$$
 
\subsection{Direct sums and entropy}
If we take two algebras $A$ and $B$ with filtrations $\{V_i\}_{i\geq 0}$ and $\{W_i\}_{i\geq 0}$ respectively, we can consider the algebra $A\oplus B$ and check that the set $\{V_i\oplus W_i\}_{i\geq 0}$ is a filtration. First of all, it is immediate that $\bigoplus_{i=0}^\infty V_i\oplus W_i = \left (\bigoplus_{i=0}^\infty V_i\right )\oplus \left (\bigoplus_{i=0}^\infty W_i\right ) = A \oplus B$. Secondly, we have, without any doubt, that $V_i\oplus W_i \subseteq V_{i+1}\oplus W_{i+1}$. Lastly, $A\cdot B = 0$, which implies $\left ( V_i\oplus W_i \right )\left (V_j\oplus W_j\right ) = V_iV_j\oplus W_iW_j \subseteq V_{i+j}\oplus W_{i+j}$.  
\begin{proposition}\label{directsums} Let $A,B$ be two algebras and $\F=\{V_i\}_{i\geq 0},\ \G=\{W_i\}_{i\geq 0}$ their respective filtrations. Consider $A \oplus B$ with the filtration $\mathcal{H}=\{V_i \oplus W_i\}_{i\geq 0}$, then
$$\h(A \oplus B)=\rm{max }\{\h(A),\h(B)\}.$$
\end{proposition}
\begin{proof}
Define $H_i:=\{V_i \oplus W_i\}$ for $i\geq 0$. First observe that for every $k$, $$\dim \left (\frac{H_{k+1}}{H_k}\right )=\dim\left (\frac{V_{k+1}\oplus W_{k+1}}{V_k \oplus W_k} \right ) =\dim\left (\frac{V_{k+1}}{V_k} \right )+\dim \left (\frac{W_{k+1}}{W_k} \right ).$$ Secondly, again take into account that for two sequences $\{a_n\}_{n\geq 0}$ and $\{b_n\}_{n\geq 0}$ with $a_n,\ b_n>0$ the following equality holds: $$\limsup_{n \to \infty} \frac{\log(a_n +b_n)}{n}=\max \left \{\limsup_{n \to \infty} \frac{\log(a_n)}{n}, \limsup_{n \to \infty}\frac{\log(b_n)}{n}\right \}.$$ So we have that
\begin{equation*}
\begin{split}
     \h(A \oplus B) & =\limsup_{n \to \infty} \frac{1}{n}\log\left (\dim \left (\frac{H_{n+1}}{H_n}\right )\right)= \\
     & =\limsup_{n \to \infty}\frac{1}{n}\log\left (\dim\left (\frac{V_{n+1}}{V_n}\right )+\dim\left (\frac{W_{n+1}}{W_n} \right )\right )= \\ & = \max\left \{\limsup_{n \to \infty} \frac{1}{n}\log\left (\frac{V_{n+1}}{V_n} \right ),\limsup_{n \to \infty} \frac{1}{n}\log\left (\frac{W_{n+1}}{W_n} \right )\right\}= \\
     & = \max\{\h(A),\h(B)\}.
\end{split}
\end{equation*}
\end{proof}
\begin{corollary}\label{cereales}
Let $E$ be a finite directed graph. If $KE$ is the corresponding path algebra with 
\begin{equation*}
    KE = \bigoplus_{i=1}^m KE_i
\end{equation*}
with $E_i$ for $i = 1,\ldots,m$ the connected components of $E$, then 
\begin{equation*}
    \h(KE) = \max \left \{\h(KE_1), \h(KE_2),\ldots, \h(KE_m) \right \}.
\end{equation*}
Observe that this result is also true for its corresponding $L_K(E)$ and $K\hat{E}$.
\end{corollary}
\subsection{Morita equivalence and entropy}
Let $(R,\F)$ be a filtered algebra with
$\F=\{V_i\}_{i\ge 0}$ and consider the algebra $M_n(R)$ which we identify with $M_n(K)\otimes R$. Then, we can define a filtration $W_i:=M_n(K)\otimes V_i$ of $M_n(R)$ such that the monomorphism
$j\colon R\to M_n(K)\otimes R$ satisfies 
$W_i\cap j(R)=V_i$. Consequently $\dim(W_k)=n^2\dim(V_k)$ and so, if we consider $M_n(R)$ endowed with the filtration $\{W_i\}$ we have
\begin{equation}\label{mor2}
\h(M_n(R))=\limsup_{k\to\infty}
\frac{\log[n^2(\dim(V_k)-\dim(V_{k-1}))]}{k}=
\end{equation}
\begin{equation*}
   =  \limsup_{k\to\infty}\frac{
\log[\dim(V_k)-\dim(V_{k-1})]}{k}=\h(R).
\end{equation*}

\begin{remark}\rm
Since the entropy of 
$K[x,x^{-1}]$ is zero,  then the entropy of $L_K(C_n)=0$ (being $C_n$ any cycle of 
length $n$), the reason of this  is that $L_K(C_n)\cong M_n(K[x,x^{-1}])$ \cite{AAS, AbramsPino}. In more detail, if the entropy of 
$K[x,x^{-1}]$ is zero, then the entropy of $M_n(K[x,x^{-1}])$ is zero and  we have an isomorphism $\varphi\colon M_n(K[x,x^{-1}]) \to L_K(C_n)$. So we know that $h_{alg}(M_n(K[x,x^{-1}]))=0$ for some filtration $\{V_n\}_{n\ge 1}$ and $\F=\{\varphi(V_n)\}_{n\ge 1}$ is a filtration of $L_K(C_n)$ with  $\h(L_K(C_n),\F)=0$, then by item (\ref{mandarina}) in Proposition \ref{sopadecebolla} for the standard filtration we have $\h(L_K(C_n))=0$. 
Observe that  standard filtrations are all of them in the hypothesis of item (\ref{mandarina}), that to say, if $\F=\{V_n\}$ is a standard filtration, then $\F$ verifies  that $V_1$ is finite dimensional and $V_k=(V_1)^k$ (for any $k$).
This also proves that 
the entropy of multi-headed comets is zero by the results of Hazrat, see Corollary 3.12 in \cite{Hazratsebandalvilela}. Also we can aboard the task of computing the entropy of
polycephaly graphs (see \cite{Hazrat}).
\end{remark}

Now, recall that two unital rings $R$ and $S$ are Morita equivalent 
if and only if there is a  full idempotent $e$ in 
the matrix ring $M_n(R)$ such that  
 $S\cong e M_n(R)e$ for some positive integer $n$, see e.g.~\cite{Lam}.
\begin{proposition}\label{full}
Assume that $(R,\F)$ is a filtered (unital) $K$-algebra (finitely generated) with $\F=\{V_i\}_{i\ge 1}$. Let $S:=eRe$ for some full
idempotent $e\in R$. Then, relative to the filtration $\G=\{eV_ie\}_{i\ge 1}$ on $S$ we have $\h(S)\le\h(R)$.
\end{proposition}
\begin{proof} For the inclusion $j\colon S\to R$ we have $V_i\cap S=W_i$. Applying Lemma \ref{cota_monomorfismo} we get $\h(S)\le\h(R)$.
\end{proof}
\begin{theorem}\label{theorem_entropy_morita}
Assume that $A$ and $B$ are Morita equivalent unital $K$-algebras and finitely generated. If $\F$ is a filtration of $A$, there is a filtration on $B$ such that relative to these, we have $\h(A)=\h(B)$. 
\end{theorem}
\begin{proof}
It suffices to prove that $\h(A)\le \h(B)$ by the symmetry of the Morita 
equivalence property. But we know that $A\cong eM_n(B)e$ for a full 
idempotent $e$ of
$M_n(B)$. Then, Proposition \ref{full} implies that $\h(eM_n(B)e)\le\h(M_n(B))$
and Formula \eqref{mor2} gives $\h(M_n(B))=\h(B)$. Thus we have $\h(A)\le\h(B)$ 
and by symmetry $\h(A)=\h(B)$.  
\end{proof}

\begin{remark}\rm
For row-finite graph $E$ and a loopless nonsink $v \in E^0$, replace each path $fg$ of length
$2$ such that $r(f) = v = s(g)$ with an edge labeled $fg $ from $s(f)$ to $r(g)$ and delete $v$
and all edges touching $v$. This is the so-called reduction algorithm which is a Leavitt path algebra Morita invariant  (see \cite{KocOzaydin2018}). Hence, by Theorem \ref{theorem_entropy_morita}, graphs obtained by reduction has Leavitt path algebras with equal entropy. However, this might not be the standard filtrations we are considering in this paper. This could be seen in Example \ref{morita_example}. Nevertheless for a source, we shall see in Corollary \ref{jamon} that the entropy using the standard filtrations is preserved. 
\end{remark}

\section{Entropy for path and Leavitt path algebras of finite graphs.}

\begin{definition}\rm Let $E$ be a finite directed graph with $n$ vertices. The \emph{adjacency matrix} of $E$ denoted by $A_E:=(a_{i,j})_{n \times n}$  where $a_{i,j}= \vert \{e \in E^1 : s(e)=v_i, r(e)=v_j\}\vert$.
\end{definition}

Consider the square matrix $A = (a_{i,j})_{1\leq i,j\leq m}$  and the matrix norm $\|A\|_{1,1}:=\sum_{i,j= 1}^m \vert a_{i,j} \vert$. This is a submultiplicative norm. Furthermore we have:

\begin{proposition}\label{tostadas} Let $A_E$ be the adjacency matrix associated to a finite directed graph $E$. Then $$\h(KE)=\displaystyle \limsup_{n\to\infty}\frac{\log(\| A_E^n \|_{1,1})}{n}.$$
\end{proposition}

\begin{proof}
If $A_E^n=(a_{i,j})$, then $a_{i,j}$ is the number of paths of length $n$ from the vertex $v_i$ to the vertex $v_j$ in the graph.
This implies that $\|A_E^n\|_{1,1} = \sum_{i,j = 1}^m  a_{i,j}  = \vert \{ \mu \in \hbox{Path}(E) \colon l(\mu) = n\} \vert  = \dim \left ( V_{n}/V_{n-1}\right ) $ with $\{V_{i}\}_{i\geq 0}$ the standard filtration of $KE$.
\end{proof}
\begin{example}\rm  Let $E$ be the following graph:

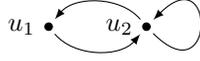
\begin{figure}[ht]
    \begin{center}
         \begin{tikzpicture}[scale = 0.65, shorten <=2pt,shorten >=2pt,>=latex, node distance={5mm},sub/.style = {draw, fill, circle, inner sep = 1pt}, main/.style = {draw, fill, circle, inner sep = 1pt}, sub/.style = {draw = red, fill = red, circle, inner sep = 1pt}]
    \node[main,label = left:$u_1$] (1) at (3,0) {};
    \node[main,label = left:$u_2$] (2) at (5,0) {};
     
    \draw[->] (1) to [bend right = 50] (2);
    \draw[->] (2) to [bend right = 50] (1);
    \draw[->] (2) to [out = 315, in = 45, looseness = 40] (2);
    \end{tikzpicture}
    \caption{Directed graph $E$}
    \label{fig_grafo_fibonacci}
    \end{center}
\end{figure}

The corresponding adjacency matrix is given by:
$$A_E= \begin{pmatrix}
0 & 1\\
1 & 1\\
\end{pmatrix}$$
It is easy to prove by induction that 
$$A_E^n= \begin{pmatrix}
f_{n-1} & f_n\\
f_n & f_{n+1}\\
\end{pmatrix},$$ where $f_0=0$, $f_1=1$, $f_2=1$, and $f_n=f_{n-1}+f_{n-2}$ for $n \geq 3$. Denote $e_+=\frac{1+\sqrt{5}}{2}$ and $e_-=\frac{1-\sqrt{5}}{2}$. Diagonalizing we have that $A_E^n=PD^nP^{-1}$ where $D=\text{diag}(e_+,e_-)$ and $P=\begin{pmatrix}
-e_- & -e_+ \\
1 & 1 \\
\end{pmatrix}$, that is, $$A_E^n=\frac{1}{\sqrt{5}}\begin{pmatrix}
e_+^{n-1}-e_-^{n-1} & e_+^n-e_-^n \\
e_+^n-e_-^n  & e_+^{n+1}-e_-^{n+1} \\
\end{pmatrix}.$$
So finally $\| A_E^n \|=\frac{1}{\sqrt{5}}(e_+^n-e_-^n  + e_+^{n+1}-e_-^{n+1})$. In order to compute $\h(KE)=\limsup_{n\to\infty}\frac{\log(\| A_E^n \|)}{n}$ we apply Stolz criterion, giving that $$\h(KE)= \lim_{n\to \infty}\log\left(\frac{e_+^{n+1}-e_-^{n+1}  + e_+^{n}-e_-^{n}}{e_+^{n}-e_-^{n}  + e_+^{n-1}-e_-^{n-1}}\right)=\log(e_+).$$
\end{example}
Taking into account that the spectral radius of a square matrix is the maximum of the modules of its eigenvalues, see \cite[18.8 Definition, p. 355]{rudin1970real}
and the formula for the spectral radius in a Banach algebra \cite[18.9 Theorem, p. 355]{rudin1970real} 
we have the following consequence:

\begin{theorem}\label{mantecao}
Let $E$ be a finite directed graph and $A_E$ its adjacency matrix. If $KE$ is the corresponding path algebra we have that the entropy of $KE$ is finite and
\begin{equation*}
    \h(KE) = \log (\rho (A_E))
\end{equation*}
with $\rho (A_E)$ the spectral radius. 
In particular, $$\h(KE)=\limsup_{n\to\infty}\frac{
\log\dim(V_n/V_{n-1})}{n}=\lim_{n\to\infty}\frac{
\log\dim(V_n/V_{n-1})}{n}$$ for  the standard filtration $\{V_n\}$. 
\end{theorem}
\begin{proof}
Since $\lim_{n\to \infty} \|A^n\|^{1/n} = \rho(A)$ for every square matrix $A$ and every matrix norm. Following Proposition \ref{tostadas}
\begin{equation*}
    \h(KE) = \displaystyle \limsup_{n\to\infty}\frac{\log(\| A_E^n \|_{1,1})}{n} = \displaystyle \limsup_{n\to\infty}\log(\left (\| A_E^n \|_{1,1}\right )^{1/n}).
\end{equation*}
\end{proof}

Consider $E = (E^0,E^1,s,r)$ a directed graph and $v \in E^0$. We denote by $E\setminus v$  the graph $F = (F^0,F^1,s,r)$ with $F^0 = E^0\setminus \{v\}$ and $F^1 = E^1 \setminus (s^{-1}(v) \cup r^{-1}(v))$. Furthermore, we can define the graph $E\setminus X$ for $X \subseteq E^0$ in the natural way. This is a generalization of the source elimination (see \cite{AAS}), also known as Move(S), see \cite{Hazratgoncalvez}.

\begin{corollary}\label{jamon}
    Let us consider $E$ a finite directed graph and $X \subseteq {\rm Sink}(E) \cup {\rm Source}(E)$. Then 
    \begin{equation*}
        \h(KE) = \h(KE \setminus X).
    \end{equation*}
\end{corollary}
\begin{proof}
We will just need to prove this for the case that $X = \{v\} \subseteq E^0$ with $v$ being a sink. If $v$ is a source is completely analogous. And the rest of the proof is an iteration of this argument. If we consider the adjacency matrix $A_E$, we have that
\begin{equation*}
    A_E = \left ( \begin{array}{cccccc}
         &   & * & & \\
         & &\vdots& & \\
         & & * && \\
        0 & \cdots& 0 & \cdots & 0 \\
            &  & * &&  &\\  
         & &  \vdots && & \\
        &  &  * & & & \\
    \end{array}\right )
\end{equation*}

with the $i$-th row of zeros corresponding to the vertex $v$. If we eliminate the $i$-th row and column, we get the adjacency matrix of $E\setminus v$. By computing the characteristic polynomial we get that 
\begin{equation*}
    \rho(A_E) = \max\{0,\rho(A_{E\setminus v})\} = \rho(A_{E\setminus v}).
\end{equation*}
Thanks to Theorem \ref{mantecao} we get that $\h(KE) = \h(KE\setminus v)$. 
\end{proof}

\begin{remark} \rm
    When you perform this elimination process for all the sinks and the sources in the original graph you will get a new one with different sources and sinks. You can iterate this process until there are not any sinks and sources in the final graph. \\
    Note that removing cycles indeed changes the entropy whenever the Condition EXC is not fulfilled. You can see this at the following examples.
\end{remark}

\noindent \begin{example}\rm
\begin{enumerate}
\item Let us now connect the two roses $R_n$ and $R_m$ of $n$ and $m$ petals with one edge. Then the graph has the adjacency matrix
$$
A= \begin{pmatrix}
n&1\\
0& m
\end{pmatrix}.
$$
The eigenvalues of the matrix are to be found on the diagonal. The spectral radius, i.e. $\lim_{n\to \infty}\|A^n\|^{\frac{1}{n}}$ is in this case the largest modulus of eigenvalue, which is $\max\{n,m\}$.
This shows  that a disjoint connection of two non-disjoint roses does not contribute to the entropy. Indeed we would have the same entropy $h_{alg}(E)=\log(\max\{n,m\})$, if the graph would just consist of the rose with the most petals.

\item If we again connect the two roses $R_n$ and $R_m$ but now with an edge from $R_n$ to $R_m$ and back. Then the graph has the adjacency matrix
$$
A= \begin{pmatrix}
n&1\\
1& m
\end{pmatrix}.
$$
The eigenvalues of the matrix are 
$$\lambda_{1,2}= \frac{n+m}{2}\pm \sqrt{\frac{n^2+m^2+4}{4}}$$
Hence since all entries are positive we obtain as largest eigenvalue $$\lambda_{1,2}= \frac{n+m}{2}+ \sqrt{\frac{n^2+m^2+4}{4}}.$$
The second edge back has now created non-disjoint cycles in which both roses are incorporated. Indeed the entropy is $h_{alg}(E)=\log(\frac{n+m}{2}+ \sqrt{\frac{n^2+m^2+4}{4}})$ and both roses contribute with their petals, as well as also the cycle in between.
\end{enumerate}
\end{example}

\begin{example}\rm Consider the $m$-rose petals graph $R_m$ with one vertex and $m$ loops. So $A_{R_m}=(m)$ and $\| A_{R_m}^n \|=m^n$. Therefore we have $\h(KR_m)=\limsup_{n\to\infty}\frac{\log(m^n)}{n}=\log(m).$
\end{example}

\textcolor{black}{We will now compute directly the entropy of the path algebras considered over the cycle $C_n$. For this we will make use of the previous results. }
\begin{example}\label{donut}\rm
Let us consider $E = C_n$ the cycle with vertices $E^0=\{v_0,v_1,\ldots,v_{n-1}\}$ and edges $E^1 = \{e_0,e_1,\ldots,e_{n-1}\}$ with $s(e_i) = v_i$ and $r(e_i) = v_{i+1}$ for $i = 0,1,\ldots, n-1$ $(\text{mod }n)$ as in Figure \ref{fig_ciclo_n}.

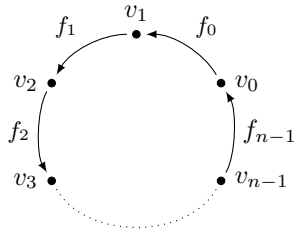
\begin{figure}[ht]
    \centering
    \begin{tikzpicture}[scale = 0.65, shorten <=2pt,shorten >=2pt,>=latex, node distance={15mm}, main/.style = {draw, fill, circle, inner sep = 1pt}]
\node[main,label = right:$v_0$] (1) at (30:\Rad) {}; 
\node[main,label = above:$v_1$] (2) at (90: \Rad) {};
\node[main, label = left:$v_2$] (3) at (150: \Rad) {};
\node[main, label = left:$v_3$] (4)  at (210: \Rad) {};
\node [main, label = right:$v_{n-1}$] (5) at (330: \Rad) {};
\draw[->] (1) to [out = 120, in = 0, looseness = 0.9] node [auto,swap] {\small{$f_{0}$}}  (2);
\draw[->] (2) to [out = 180, in = 60, looseness = 0.9] node [auto,swap] {\small{\small{$f_{1}$}}} (3) ;
\draw[->] (3) to [out = 240, in = 120, looseness = 0.75] node [auto,swap] {\small{\small{$f_{2}$}}}(4); 
\draw[dotted] (4) to [out = 300, in = 240, looseness = 1] (5);
\draw[->] (5) to [out = 60, in = 300, looseness = 0.75] node [auto,swap] {$f_{n-1}$}   (1);
\end{tikzpicture}
    \caption{Directed cycle with $n$ vertices.}
    \label{fig_ciclo_n}
\end{figure}

If we consider its corresponding adjacency matrix
\begin{equation*}
    \left ( \begin{array}{ccccc}
         0& 1 & 0 & \cdots & 0  \\
         0 & 0 & 1 & \cdots & 0 \\
         \vdots & \vdots & \vdots & \ddots & \vdots \\
         0 & 0 & 0 & \cdots & 1 \\ 
         1 & 0 & 0 & \cdots & 0 \\
    \end{array}\right ).
\end{equation*}
This matrix verify that $\|A_E^m\| = n$ for all $m \geq 0$. Then 
\begin{equation*}
    \h(KE) = \limsup_{m \to \infty} \frac{\log(n)}{n} = 0. 
\end{equation*}
If we consider $\hat{E}$, the extended graph of $C_n$, the adjacency matrix $A_{\hat{E}}$ is
\begin{equation*}
    \left ( \begin{array}{cccccc}
         0& 1 & 0 & \cdots & 0 & 1 \\
         1& 0 & 1 & \cdots& 0& 0 \\
         0 & 1 & 0 & \cdots &0 & 0 \\
         \vdots & \vdots & \vdots & \ddots & \vdots & \vdots \\
         0 & 0 & 0 & \cdots& 0& 1  \\ 
         1 & 0 & 0 & \cdots& 1& 0 \\
    \end{array}\right ).
\end{equation*}
This is a circulant matrix with eigenvalues $\lambda_j = \omega^j+\omega^{(n-1)j} = e^{i\frac{2\pi j}{n}} + e^{-i\frac{2\pi j}{n}} = 2 \cos(\frac{2\pi j }{n})$ for $j = 0,1, \ldots, n-1$ (\cite[Theorem 3.2.2]{circulant}). Then $\rho(A_{\hat{E}} ) = 2$ and thanks to Theorem \ref{mantecao} we have that $\h(K\hat{E}) = \log(2)$. We can check that Inequation \eqref{eqone} holds $0 \leq \log (2)$.
\\
In order to calculate the entropy of the Leavitt path algebra $L_K(E)$ we need to know the standard filtration. If we take into account that 
\begin{equation*}
\begin{split}
    e_ie_j & = \delta_{i+1,j}e_ie_{i+1}, \\
    e_j^*e_i^* & = \delta_{j,i+1} e_{i+1}^*e_i^*, \\
    e_ie_j^* & = \delta_{i,j} v_i,\\
    e_i^*e_j & = \delta_{i,j} v_{i+1}, 
\end{split}
\end{equation*}
the computation of $V_i$ will be easier. We have that
\begin{equation*}
\begin{split}
        V_0 & = \span(E^0),\\
        V_1 & = V_0 + \span(E^1 \cup (E^1)^*), \\
        V_2 & = V_1 + \span\{e_ie_{i+1}, e_{i+1}^*e_i^*, \text{ for } i= 0,\ldots,n-1 (\text{mod } n)\}, \\
        & \vdots \\
        V_k & = V_{k-1} + \span\{e_ie_{i+1}\cdots e_{i+k}, e_{i+k}^*\cdots e_i^*, \text{ for } i = 0,1, \ldots, n-1 (\text{mod } n)\}.
\end{split}
\end{equation*}
Thus, the entropy is
\begin{equation*}
    \h(L_K(E)) = \limsup_{k\to \infty} \frac{\log \left ( \frac{\dim \left ( V_{k+1} \right )}{\dim (V_k)}\right )}{k} = \limsup_{k\to \infty} \frac{\log \left ( 2n\right )}{k} = 0.
\end{equation*}
Once more, Inequality \eqref{eqone} holds $0 \leq 0 \leq \log(2)$. 
\end{example}

Note that the result in the previous example coincides with the following corollary and the later results of the section.

For $E$ a finite directed graph, the adjacency matrix $A_{\hat{E}}$ coincides with $A_E+A_E^t$, where $t$ denotes its transpose.
\begin{corollary}\label{log2} Let $E$ be a finite directed graph, then $\h(K\hat{E})=\log(\rho(A_E + A_E^t))$. Moreover, if $A_E$ is normal, that is, $A_E \cdot A_E^t=A_E^t \cdot A_E$, then $$\h(K\hat{E}) \leq \h(KE) +\log(2) \text{ and}$$
$$\h(L_K(E)) \leq \h(KE) +\log(2).$$
\end{corollary}
\begin{proof}The first inequality, it is easy to check taking into account that if  $A$ and  $B$ are  bounded linear  operators mapping a Banach space $(X, \vert \vert \cdot \vert \vert)$ into itself and these operators are commutative  (by Equation 21 p. 426 of \cite{Riesz}), then $\rho(A+B)\le \rho(A)+\rho(B)$. The second inequality is a consequence of \eqref{eqone}.
\end{proof}
\begin{remark}\rm Corollary \ref{log2} holds if $A_E$ is a symmetric matrix, in particular for $A_{R_n}$ with $R_n$ the $n$-rose petals graph. Also for $A_{C_n}$ with $C_n$ a cycle with $n$ vertices.
\end{remark}
\bigskip

\begin{lemma}\label{cumplealfi}
    Let $E$ be a finite directed graph satisfying Condition (EXC) and without sources and sinks. Then $\h(KE) = 0$.
\end{lemma}
\begin{proof}
   It will only be necessary to prove this result for connected graphs, the case of non-connected graphs is a direct consequence of Corollary \ref{cereales}.  We will proceed to make our proof by induction on the number of cycles. If $E$ has one cycle without sinks and sources, then $E = C_n$ and we have proven in Example \ref{donut} that $\h(KE) = \h(KC_n) = 0$. Let us assume that this result is true for $k$ cycles. Let us consider a graph $E$ with $k+1$ cycles without sources or sinks. This implies that there is at least one cycle without entries or a cycle without exits. Without loss of generality, suppose there is a cycle $C_n$ without entries. For a suitable order of the vertices, the adjacency matrix of $E$ is:    

    \begin{equation*}
        A_E = \left ( \begin{array}{cc}
            A_{C_n} & * \\
            0 & A_F
        \end{array} \right )
    \end{equation*}
    with $F = E\setminus C_n^0$. Accordingly, $\rho(A_E) = \max\{ \rho(C_n),\rho(A_F)\} = \max\{1,\rho(A_F)\}$. The graph $F$ has $k$ cycles, but we may have sinks or sources. However, Corollary \ref{jamon} implies that we get the same entropy even if eliminate all the sources and sinks. Now we are under our induction hypothesis so $\h(KF) = 0$ and this means by Theorem \ref{mantecao} that $\rho(A_F) = 1$. Finally, $\rho(A_E) = 1$ and $\h(KE) = 0$.

\end{proof}

Summarizing we obtain:
\begin{theorem} \label{hihi}
Let $E$ be a finite directed graph and $KE$ the associated path algebra, then 
\begin{itemize}
    \item[i)] $\gkdim(KE) =  0$ if and only if $KE$ is finite-dimensional;
    \item[ii)]  If $0\neq  \gkdim(KE)<\infty$, then $\h(KE)=0$ and $KE$ is infinite dimensional.
\end{itemize}
\end{theorem}

\begin{proof} For the first statement apply \cite[Theorem 3.12]{moreno2018graph}. Now we prove item (ii). If $0<\gkdim(KE) = k < \infty$, then by \cite[Theorem 3.12]{moreno2018graph}  we know that the maximal length of chains of cycles in $E$ is $k$. By removing subsequently all sinks and all sources we get a graph with the same entropy (Corollary \ref{jamon}) and by Lemma \ref{cumplealfi}, each connected component has null entropy. Finally, by Corollary \ref{cereales}, we obtain $\h(E) = 0$. Moreover, the dimension of $KE$ is infinite since the graph $E$ has at least one cycle. 
\end{proof}

\begin{lemma}
 Let $A$ be a filtered algebra with filtration $\{V_n\}_{n >0}$ such that $h_{alg}(A)= \infty$, then $\dim(V_n/V_{n-1})$ grows superexponential, i.e.~
 $$\limsup_{n\to \infty} \frac{\dim(V_n/V_{n-1})}{c^n} = \infty,\quad \text{ for any } c>0.$$
\end{lemma}
\begin{proof}
Under the hypothesis we have that $\limsup_{n\to\infty}\frac{
\log\dim(V_n/V_{n-1})}{n}=\infty.$
Then for any $c>0$, there exists $n$ with 
$$\frac{
\log\dim(V_n/V_{n-1})}{n}>\log c.$$
Accordingly, we have $\log \dim(V_n/V_{n-1}) > n\log c=\log c^n$. Now, since the exponential function is monotonic increasing, we obtain 
$\dim(V_n/V_{n-1})>c^n.$ 
The conclusion follows.
\end{proof}

\begin{definition}\label{def:growthclass}
For a filtered  $K$-algebra $A$ we say:  \begin{itemize}
    \item $A$ is of Class $0$ if $A$ is finite dimensional. 
    \item $A$ is of Class $1$ if $A$ is  infinite dimensional with $\gkdim(A)<\infty$. 
    \item $A$ is of Class $2$ if $A$ $\gkdim(A)=\infty$ and $h_{alg}(A)< \infty$.
\end{itemize}
\end{definition}
\begin{theorem}[Growth Trichotomy]
Let $E$ be a finite graph. For $A=KE$ or $A=L_K(E)$ we have that $A$ is in exactly in one growth classes from Definition \ref{def:growthclass}.
 In particular, we have  $h_{alg}(A)<\infty$.
\end{theorem}
\begin{proof}
This is a direct consequence of fact that a finite graph has a finite adjacency matrix $A$ and hence the spectral radius is finite for all $A^n$. This leads to a finite or vanishing algebraic entropy by Theorem \ref{mantecao}. The claim then follows by Theorem \ref{hihi}.
\end{proof}

As a consequence $A=L_K(E)$, or $A=KE$ where $E$ is a finite graph, we can
   construct a triple $t(A,\F):=(\dim(A),\gkdim(A),\h(A))$ associated to the filtered algebra $(A,\F)$.  So $t$ is a map from the class of objects of the category of filtered algebras to $\overline{\mathbb{R}}^3$ where $\overline{\mathbb{R}}=\mathbb{R}\cup\{\infty\}$.
   This is an invariant since any isomorphism in the category of filtered algebras induces an equality of the corresponding 
   triples. All the (filtered) algebras in Class 1 are classified by Gelfand-Kirillov dimension since they all have $0$ entropy. To be more precise any two algebras $A$ and $B$ of Class 1 with $t(A)\ne t(B)$ can not be isomorphic. 
On the other hand, the algebras in class 2 have the same Gelfand-Kirillov dimension infinity, and they can be distinguished by entropy. The growth trichotomy then reads as the following corollary.

\begin{corollary} Let $A= KE$ or $L_K(E)$ for a finite graph $E$. Then $A$ can be classified into three types as follows: if $t(A,\F):=(\dim(A), \gkdim(A), \h(A))$ one has
\begin{enumerate}
    \item $t(A,\F)=(k,0,0)$ for $k < \infty$; or
    \item $t(A,\F)=(\infty,l,0)$ for $l<\infty$; or
    \item $t(A,\F)=(\infty,\infty,m)$ for $m<\infty$.
\end{enumerate}
\end{corollary}

\section{Some examples computing the algebraic entropy of $L_K(E)$.}
In this section we analyze some examples in which we compute $\h(L_K(E))$ for certain graphs $E$. This gives us evidence of how close the numbers $\h(KE)$ and $\h(L_K(E))$ are. First, we need to count the number of linearly independent elements $\lambda \mu^*$ in $L_K(E)$ with $l(\lambda)+l(\mu) \leq k \in \mathbb{N}$, $\lambda,\mu \in {\rm Path}(E)$. This formula is a consequence of Theorem 38 of \cite{Bocksebandal}.

\begin{proposition} Let $E$ be a finite graph, $A_E$ its corresponding adjacency matrix and fix $k \in \mathbb{N}$. In $L_K(E)$ the number of linearly independent elements of the form $\lambda \mu^*$ such that $l(\lambda)+l(\mu)=k$, $\lambda,\mu \in {\rm Path}(E)$, that is $\dim(V_k/V_{k-1})$, is equal to:
\begin{equation}\label{countpaths}
    \sum_{s=0,j=1}^{k,n} \left (\sum_{i=1}^n (A_E^s)_{i,j} \right ) \left (\sum_{i=1}^n(A_E^{k-s})_{i,j} \right )- \sum_{s=1,j=1}^{k-1,n} \left ( \sum_{i=1}^n (A_E^{s-1})_{i,j}\right ) \left ( \sum_{i=1}^n (A_E^{k-s-1})_{i,j} \right )\gamma_{j}
\end{equation}
where $\gamma_j = 0$ if $\sum_{m=1}^n (A_E)_{j,m} = 0$, else $\gamma_j =1$.
\end{proposition}

We include the source code implemented with  Mathematica software ({\it Mathematica-Wolfram Research, Inc.})

{\small

\begin{verbatim}
    p[A_, m_, a_, b_] :=
  Module[{n}, n = Length[A]; 
   If[m == 0, IdentityMatrix[n][[a, b]], (MatrixPower[A, m])[[a, b]]]];
\end{verbatim} }

The above command {\tt p[A,m,a,b]} computes the $(a,b)$-entry of the $m$-th power of the matrix $A$. The next one, {\tt cond[A,j]}, codifies the $\gamma_j$ function defined in \eqref{countpaths}. 

{\small
\begin{verbatim}
cond[A_, j_] := Module[{n}, n = Length[A]; Sign[Sum[p[A, 1, j, m], {m, n}]]];
\end{verbatim}}

The function {\tt aux[A,j,s,n]} computes the second part 
in \eqref{countpaths} when $\gamma_j=1$, that is, when {\tt cond[A,J]}
evaluates to \lq\lq TRUE\rq\rq.

{\small
\begin{verbatim}
aux[A_, j_, s_, k_] := 
 Module[{n}, n = Length[A]; 
  Sum[p[A, s - 1, i, j], {i, n}] Sum[p[A, k - s - 1, i, j], {i, n}]]; 
  \end{verbatim}}

And finally {\tt h[A,k]} computes the algebraic entropy of the Leavitt path algebra with adjacency matrix $A$, based on the computations of linearly independent elements $\lambda\mu^*$ given by equation \eqref{countpaths}.

{\small
\begin{verbatim}
h[A_, k_] := N[Log[Module[{n}, n = Length[A]; Sum[
       Sum[
        Sum[p[A, s, i, j], {i, n}] Sum[p[A, k - s, i, j], {i, n}], {j,
          n}], {s, 1, k - 1}] - 
      Sum[Sum[aux[A, j, s, k] cond[A, j], {j, n}], {s, 1, k - 1}] + 
      2 Sum[p[A, k, i, j], {i, n}, {j, n}]]]/k];
\end{verbatim}}

By using the above Mathematica code, all the examples considered, seem to suggest the coincidence of both numbers $\h(L_K(E))$ and $\h(KE)$.

\begin{example}\label{otrografito}\rm Consider the following graph $E$:

\begin{figure}[H]\label{otroejemplo}
    \begin{center}
         \begin{tikzpicture}[scale = 0.65, shorten <=2pt,shorten >=2pt,>=latex, node distance={5mm},sub/.style = {draw, fill, circle, inner sep = 1pt}, main/.style = {draw, fill, circle, inner sep = 1pt}, sub/.style = {draw = red, fill = red, circle, inner sep = 1pt}]
    \node[main,label = left:$\tiny 1 $] (1) at (3,0) {};
    \node[main,label = above:$2$] (2) at (5.8,0) {};
    \node[main,label = above:$3$] (3) at (9,1.2) {};
    \node[main,label = below: $4$] (4) at (9,-1.2) {};
     
    \draw[->] (1) to [bend right = 50] (2);
    \draw[->] (2) to [bend right = 50] (1);
    \draw[->] (3) to (2);
    \draw[->] (4) to (2);
     \draw[->] (3) to [bend right = 50] (4);
     \draw[->] (4) to [bend right = 50] (3);
    \end{tikzpicture}
    \caption{Directed graph $E$ in Example \ref{otrografito}}
    \end{center}
\end{figure}
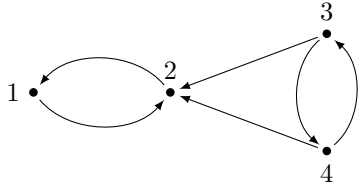

Taking into account formula (\ref{countpaths}), computing $\displaystyle\frac{\log(\dim(V_k/V_{k-1}))}{k}$ is the best numerical approximation to $\h(L_K(E))$. For instance,  we get with {\it Mathematica} that $$\displaystyle \frac{\log(\dim(V_{1000}/V_{999}))}{1000}=0.0145107.$$ On the other hand, the logarithm of the spectral radius (the maximum of the absolute values of the eigenvalues of $A_E$) equals $0$, having $\h(KE)=0$ by Theorem \ref{mantecao}. Increasing the value of $k$ we approximate more closely to $0$. 
\end{example}

\begin{example} \label{villancico}
\rm Next we take the graphs $F_1$ and $F_2$ given by:
\begin{figure}[ht]
    \begin{center}
         \begin{tikzpicture}[scale = 0.65, shorten <=2pt,shorten >=2pt,>=latex, node distance={5mm},sub/.style = {draw, fill, circle, inner sep = 1pt}, main/.style = {draw, fill, circle, inner sep = 1pt}, sub/.style = {draw = red, fill = red, circle, inner sep = 1pt}]
    \node[main,label = left:$1$] (1) at (3,0) {};
    \node[main,label = above:$2$] (2) at (5.5,0) {};
    \draw[->] (2) to  (1);
    \draw[->] (1) to  [out = 55, in = 135, looseness = 40] node[auto] {} (1);
    \draw[->] (1) to  [out = 220, in = 300, looseness = 40] node[auto] {} (1);
    \draw[->] (2) to  [out = 35, in = 115, looseness = 40] node[auto] {} (2);
    \draw[->] (2) to  [out = 35, in = 315, looseness = 40] node[auto] {} (2);
    \draw[->] (2) to  [out = 315, in = 235, looseness = 40] node[auto] {} (2);
    \end{tikzpicture}\ \ \begin{tikzpicture}[scale = 0.65, shorten <=2pt,shorten >=2pt,>=latex, node distance={5mm},sub/.style = {draw, fill, circle, inner sep = 1pt}, main/.style = {draw, fill, circle, inner sep = 1pt}, sub/.style = {draw = red, fill = red, circle, inner sep = 1pt}]
    \node[main,label = left:$1$] (1) at (3,0) {};
    \node[main,label = above:$2$] (2) at (5.5,0) {};
    \draw[->] (2) to  (1);
    \draw[->] (2) to  [out = 45, in = 125, looseness = 40] node[auto] {} (2);
    \draw[->] (2) to  [out = 320, in = 240, looseness = 40] node[auto] {} (2);
    \draw[->] (1) to  [out = 55, in = 135, looseness = 40] node[auto] {} (1);
    \draw[->] (1) to  [out = 145, in = 225, looseness = 40] node[auto] {} (1);
    \draw[->] (1) to  [out = 225, in = 305, looseness = 40] node[auto] {} (1);
    \end{tikzpicture}
    \caption{Directed graphs $F_1$ and $F_2$ in Example \ref{villancico}}
    \label{otroejemplo2}
    \end{center}
\end{figure}
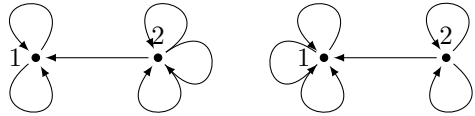

Again by formula \ref{countpaths}, according to the calculations made with {\it Mathematica} we have that $\displaystyle \frac{\log(\dim(V_{1000}/V_{999}))}{1000}=1.1061$ and $\h(KF_1)=\h(KF_2)=\log(3)=1.09861$. We observe again the same phenomenom. By the way, notice that it seems to be independent of the orientation of the edges.
\end{example}

\begin{example} \label{mortadela}
\rm Now suppose the graph $G$ is the one given in Figure \ref{alegria}.

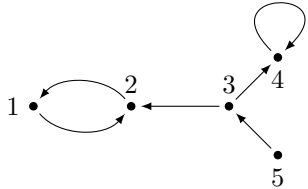
\begin{figure}[h]
    \begin{center}
         \begin{tikzpicture}[scale = 0.65, shorten <=2pt,shorten >=2pt,>=latex, node distance={5mm},sub/.style = {draw, fill, circle, inner sep = 1pt}, main/.style = {draw, fill, circle, inner sep = 1pt}, sub/.style = {draw = red, fill = red, circle, inner sep = 1pt}]
    \node[main,label = left:$1$] (1) at (3,0) {};
    \node[main,label = above:$2$] (2) at (5,0) {};
    \node[main,label = above:$3$] (3) at (7,0) {};
    \node[main,label = below: $4$] (4) at (8,1) {};
    \node[main,label = below: $5$] (5) at (8,-1) {};
    \draw[->] (1) to [bend right = 50] (2);
    \draw[->] (2) to [bend right = 50] (1);
    \draw[->] (3) to (2);
    \draw[->] (3) to (4);
    \draw[->] (5) to (3);
    \draw[->] (4) to  [out = 135, in = 45, looseness = 40] node[auto] {} (4);
    \end{tikzpicture}
    \caption{Directed graph $G$ in Example \ref{mortadela}}
    \label{alegria}
    \end{center}
\end{figure}

In this case, for $k=1000$ we obtain  $\displaystyle \frac{\log(\dim(V_{k}/V_{k-1}))}{k}=0.00352636$ and $\h(KG)=0$ having convergence once more.
\end{example}

\begin{example}\label{tatuaje}
\rm Let $D$ be the following graph:

\begin{figure}[H]
    \begin{center}
         \begin{tikzpicture}[scale = 0.65, shorten <=2pt,shorten >=2pt,>=latex, node distance={5mm},sub/.style = {draw, fill, circle, inner sep = 1pt}, main/.style = {draw, fill, circle, inner sep = 1pt}, sub/.style = {draw = red, fill = red, circle, inner sep = 1pt}]
    \node[main,label = above:$1$] (1) at (0,0) {};
    \node[main,label = above:$2$] (2) at (90: \Rad) {};
    \node[main,label = left:$3$] (3) at (210: \Rad) {};
    \node[main,label = right:$4$] (4) at (333: \Rad) {};
     
    \draw[->] (1) to [bend right = 50] (2);
    \draw[->] (2) to [bend right = 50] (1);
    \draw[->] (1) to [bend right = 50] (3);
    \draw[->] (3) to [bend right = 50] (1);
    \draw[->] (4) to [bend right = 50] (1);
    \draw[->] (1) to [bend right = 50] (4);
    \draw[->] (2) to [bend right = 60] (3);
    \draw[->] (3) to [bend right = 60] (4);
    \draw[->] (4) to [bend right = 60] (2);
    \end{tikzpicture}
    \caption{Directed graph $D$ in Example \ref{tatuaje}}
    \end{center}
   
    \end{figure}
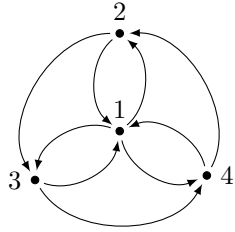

For $k=1000$, {\it Mathematica} gives us $\displaystyle \frac{\log(\dim(V_{k}/V_{k-1}))}{k}=0.842187$ and $\h(KD)=\log\left (\frac{1+\sqrt{13}}{2} \right )\approx 0.834115$.
\end{example}

\begin{example} \label{morita_example} \rm Consider the same graph as in Figure \ref{fig_grafo_fibonacci}, that is:
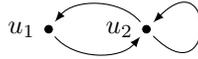
\begin{figure*}[ht]
    \begin{center}
         \begin{tikzpicture}[scale = 0.65, shorten <=2pt,shorten >=2pt,>=latex, node distance={5mm},sub/.style = {draw, fill, circle, inner sep = 1pt}, main/.style = {draw, fill, circle, inner sep = 1pt}, sub/.style = {draw = red, fill = red, circle, inner sep = 1pt}]
    \node[main,label = left:$u_1$] (1) at (3,0) {};
    \node[main,label = left:$u_2$] (2) at (5,0) {};
     
    \draw[->] (1) to [bend right = 50] (2);
    \draw[->] (2) to [bend right = 50] (1);
    \draw[->] (2) to [out = 315, in = 45, looseness = 40] (2);
    \end{tikzpicture}
    \end{center}
    \end{figure*}

Computing $\h(L_K(E))$ with {\it Mathematica} we have numerical convergence to $\log(e_+)$ where $e_+=\frac{1+\sqrt{5}}{2}$. At this point we notice that, despite the fact that $L_K(E)$ is morita equivalent to $L_K(R_2)$ (see \cite{Koc}), $\h(L_K(E))$ is (numerically) different from $\h(L_K(R_2)) = \log(2)$ (Example \ref{cena_navidad}). Observe that this does not contradict Theorem \ref{theorem_entropy_morita} since both $L_K(E)$ and $L_K(R_2)$ are considered with the standard filtrations.
\end{example}

{\noindent \bf Conclusion.} In all the above computations we obtain that $h_{alg}(L_K(E))\approx h_{alg}(KE)$. In a forthcoming work we are going to take a deeper look into this fact.

\section*{Acknowledgements}{The second, third, fourth  and fifth authors are supported by the Spanish Ministerio de Ciencia e Innovaci\'on   through project  PID2019-104236GB-I00/AEI/10.13039/\- 501100011033 and by the Junta de Andaluc\'{i}a  through projects  FQM-336 and UMA18-FEDERJA-119,  all of them with FEDER funds. The fifth author is supported by a Junta de Andalucía PID fellowship no. PREDOC\_00029. The computations were performed in the Picasso Supercomputer at the University
of Málaga, a node of the Spanish Supercomputing Network. The sixth author is supported by the Philippine 
Department of Science and Technology under the Accelerated Science and Technology Human Resource Development Program. The sixth author gratefully acknowledges the hospitality during her research stay at University of Málaga funded by Departament of Algebra, Geometry and Topology of the University of Málaga. } 

\bibliographystyle{plain}
\bibliography{ref}
\end{document}